\documentclass[11pt,a4paper]{article}
 \usepackage{euscript}
\usepackage{amsmath}
\usepackage{graphicx}
\usepackage{amsthm}
\usepackage{amssymb}
\usepackage{epsfig}
\usepackage{url}
\usepackage{colonequals}
\usepackage{amsmath,amscd}
\theoremstyle{theorem}
\newtheorem{theorem}{Theorem}[section]
\newtheorem{corollary}[theorem]{Corollary}

\newtheorem{lemma}[theorem]{Lemma}
\newtheorem{proposition}[theorem]{Proposition}

\newtheorem{definition}{Definition}[section]

\newtheorem{example}{Example}[section]
\newtheorem{remark}{Remark}[section]
\numberwithin{equation}{section}
 
\newcommand{\s}{\sigma}

\title{Amitsur--Small Extensions and a Skew Amitsur--Small Theorem} 
\author{Masood Aryapoor\\
	\tiny{\textit{Division of Mathematics and Physics}}\\
	\tiny{\textit{M\"{a}lardalen  University}}\\
	\tiny{\textit{Hamngatan 15, 632 17, Eskilstuna, 
			Sweden
	}}
}
 \date{}

\begin{document}
 \maketitle
 
 \begin{abstract}
 	
 	We introduce the notion of Amitsur--Small extensions to generalize a key lemma underlying the Amitsur--Small Theorem to the skew setting. Building on this framework, we establish a skew version of the Amitsur--Small Theorem.
 	
 \end{abstract}
\section{Introduction} 
	 In \cite{AmitsurSmall}, Amitsur and Small established the following result:
	 \begin{theorem}[Amitsur--Small Theorem]\label{thm: Weak Nullstellensatz over division rings}\label{thm: amitsur small}
	 	If $D$ is a division ring, then simple $D[X_1,\dots,X_n]$-modules are finite-dimensional as  vector spaces over $D$.
	 \end{theorem}  
	 Here, the notation $D[X_1,\dots,X_n]$ denotes for the polynomial ring over \(D\) in \(n\) central indeterminates. Their proof relies on three lemmas, one of which states:
	 \begin{lemma}[Lemma C in \cite{AmitsurSmall}]\label{lem: lemmaC}
	 	If \(L\) is a maximal left ideal of \(D[X_1,\dots,X_n]\), then \(L\cap D[X_k]\neq 0 \) for all \(k\). 
	 \end{lemma}
	 Our primary goal is to extend this lemma to the skew setting and thereby obtain a skew version of the Amitsur--Small Theorem. In doing so, we introduce a special class of ring extensions, which we term  \emph{Amitsur--Small extensions}   (see Definition \ref{def: amitsur-small-extensions}). Our definition of  Amitsur--Small extensions depends on the notion of \emph{skew polynomial rings in several variables}, which serves as a suitable generalization of polynomial rings for our purposes (see Definition \ref{def: amitsur-small-extensions}). It is worth noting that for a maximal left ideal \(L\) in  \(D[X_1,\dots,X_n]\), the left ideal \(L\cap D[X_k]\) need not be  maximal in  \(D[X_k]\).   A concrete example illustrating this phenomenon was recently constructed by  Chapman and Paran in \cite{chapman2025amitsur}, where  they introduced and studied the concept of an Amitsur--Small ring.

	 The paper is organized as follows. Section \ref{sec: skew-several} introduces the notion of skew polynomial rings in several variables and explores their relationship to iterated skew polynomial rings and \(\sigma\)-PBW extensions.   In Section \ref{sec: amitsur-small}, we define the concept of an Amitsur--Small extension and prove that if \(R\subset S\) is an Amitsur--Small extension, then for any maximal left ideal \(L\) of \(S\), we have \(L\cap R\neq 0\) (see Proposition \ref{prop: ASextensions}).  The final section presents some examples of Amitsur--Small extensions, and in particular, establishes the following skew version of the Amitsur--Small Theorem (see Theorem \ref{thm: skew-AS}):
	 \begin{theorem}[Skew Amitsur--Small Theorem]
	 	Let $D$ be a division ring. Let $\s:D\to D$ be an automorphism and $\delta:D\to D$ be  a $\s$-derivation such that \(\sigma\delta = \delta\sigma\).  Consider the ring \(D[x_1,\dots,x_n;\sigma,\delta]\), generated over \(D\) by \(x_1,\dots,x_n\), subject to the relations \(x_ix_j=x_jx_i\), \(x_ia=\s(a)x_i+\delta(a)\) for all \(i,j\) and \(a\in D\). If the center of the skew polynomial ring $D[x;\s,\delta]$ contains a nonconstant polynomial, then any simple \(D[x_1,\dots,x_n;\sigma,\delta]\)-module is finite-dimensional as  a vector space over $D$.
	 \end{theorem}  
	 The final section also presents a generalization of Lemma \ref{lem: lemmaC} (see Proposition \ref{prop: lemmaC}):
	 	\begin{proposition}
	 	Let $R$ be a principal (left and right) ideal ring with infinitely many distinct maximal ideals. Then for every maximal left ideal $L$ of the polynomial ring $R[x_1,\dots,x_n]$ over \(R\) in central indeterminates, we have $L\cap R\neq 0$. 
	 \end{proposition}
	Throughout this paper, all rings are assumed to be unital and associative, though not necessarily commutative. 
\section{Skew Polynomial Rings in Several Variables}\label{sec: skew-several}
	To generalize the Amitsur--Small theorem,  we must consider polynomial rings beyond those in central indeterminates. This section introduces the class of \emph{skew polynomial rings in several variables}, which serves as a suitable framework for such generalizations.

		Let $R$ be a subring of a ring $S$. An (ordered) sequence of elements $x_1,\dots,x_n$ in \(S\) is called \textit{(left) algebraically independent over} $R$ if the power products ${\mathbf{x}}^{\mathbf{i}}\colonequals x_1^{i_1}\cdots x_n^{i_n}$, where $\mathbf{i}=(i_1,\dots,i_n)\in \mathbb{N}^n$, are left $R$-linearly independent. We emphasize that the definition depends on the order of the elements  $x_1,\dots,x_n$.  Let $x_1,\dots,x_n\in S$ be algebraically independent over $R$.  Recall that a \textit{term order} is  a well-order $<$ on  \(\mathbb{N}^n\)  such that (1) $(0,\dots,0)<\mathbf{i}$ for all $(0,\dots,0)\neq \mathbf{i} \in \mathbb{N}^n$, and (2) if $\mathbf{i}<\mathbf{j}$, then $\mathbf{i}+\mathbf{k} < \mathbf{j}+\mathbf{k}$ for all  $\mathbf{k} \in \mathbb{N}^n$.  We fix a term order $<$ on \(\mathbb{N}^n\).  The notation \(\mathbf{x}^{\mathbf{i}} < \mathbf{x}^{\mathbf{j}}\) is sometimes used in place of \({\mathbf{i}} < {\mathbf{j}}\).
		Given a nonzero element 
		$$f=r_{1}{\mathbf{x}}^{\mathbf{i}_1}+\cdots+r_{m}{\mathbf{x}}^{\mathbf{i}_m}\in S,$$
		where $0\neq r_i\in R$ and $\mathbf{i}_1 <...<\mathbf{i}_m$, we define: $\operatorname{lt}(f)\colonequals r_{m}{\mathbf{x}}^{\mathbf{i}_m}$, called the \textit{leading term} of $f$  ;  $\operatorname{lp}(f)\colonequals{\mathbf{x}}^{\mathbf{i}_m}$, called the \textit{leading power product}  of $f$; $\operatorname{le}(f)\colonequals \mathbf{i}_m$, called the \textit{leading exponent} of $f$. We use the convention \(\operatorname{le}(0) = (-\infty,\dots,-\infty)\). 
		\begin{definition}\label{def: skew-several}
			Let $S$ be a ring and $R$ be a subring of $S$. We say that $S$ is a \emph{skew polynomial ring over $R$ in variables $x_1,\dots,x_n$}    if the following conditions hold:\\
			1) $S$ is a free left $R$-module with basis ${\mathbf{x}}^{\mathbf{i}}$, where $\mathbf{i}\in \mathbb{N}^n$.  \\
			2) There exists a term order $<$ on  $\mathbb{N}^n$ with respect to which
			\[
			\operatorname{le}(fg)=\operatorname{le}(f)+\operatorname{le}(g)
			\]
			for all $f,g\in S\setminus\{0\}$. \\
			Under these conditions, we write $S=R[x_1,\dots,x_n;<]$. 
		\end{definition}
		
		Before we present some elementary properties of skew polynomial rings in several variables, let us give some examples of such rings. 
		\begin{example}\label{exam: iterated-skew}
			Let $R$ be an integral domain and 
			$$S=R[x_1;\s_1,\delta_1]\dots [x_n;\s_n,\delta_n]$$
			be an iterated skew polynomial ring over $R$ such that every $\s_i$ is injective. Assume furthermore that $\s_i(R)\subseteq R$, $\delta_i(R)\subseteq  R$, and $\s_i(x_j)\in (R\setminus\{0\})x_j+R$  for all $i>j$. Clearly, $x_1,\dots,x_n$ are (left) algebraically independent over $R$. Let $<_l$ be the lexicographical order on  $\mathbb{N}^n$ with $x_1<...<x_n$, i.e., 
			$$x_1^{i_1}\cdots x_n^{i_n}<_lx_1^{j_1}\cdots x_n^{j_n}\iff i_n=j_n,\dots,i_{p+1}=j_{p+1}, i_p<j_p, \text{ for some } p.$$ 
			Then  $S=R[x_1,\dots,x_n;<_l]$ is a skew polynomial ring over $R$ in $x_1,\dots,x_n$. In the case where  $n=1$, every skew polynomial ring  $S=R[x;<_l]$ over $R$ in $x$ is a skew polynomial ring in the usual sense, that is, $S=R[x;\s,\delta]$ for a monomorphism $\s:R\to R$ and a $\s$-derivation $\delta:R\to R$. This follows from the fact that the only term order on \(\mathbb{N}\) is the order 
			$1<2<3<\cdots$.
		\end{example}
		The next example shows that every $\s$-PBW extension is a skew polynomial ring in several variables. For a detailed study of $\s$-PBW extensions, see \cite{fajardo2020skew}.
		\begin{example}
			Let $R$ be an integral domain and 
			$S=\s(R)\langle x_1,\dots,x_n\rangle$ be a $\s$-PBW extension of $R$. 
			Let  $<_d$ be the degree lexicographical order on the set of all power products $x_1^{i_1}\cdots x_n^{i_n}$ with $x_1<...<x_n$, i.e., 
			$x_1^{i_1}\cdots x_n^{i_n}<_dx_1^{j_1}\cdots x_n^{j_n}$ if and only if $ i_1+\cdots+i_n<j_1+\cdots+j_n$,  or $i_1+\cdots+i_n=j_1+\cdots+j_n$ and $x_1^{i_1}\cdots x_n^{i_n}<_lx_1^{j_1}\cdots x_n^{j_n}$.
			Then  $S=R[x_1,\dots,x_n;<_d]$  is a skew polynomial ring over $R$ in $x_1,\dots,x_n$.
		\end{example}
		In the following proposition, we collect some properties of skew polynomial rings in several variables.
		\begin{proposition}\label{prop: simple-proerties-skew}
			Let $S=R[x_1,\dots,x_n;<]$ be a skew polynomial ring  over $R$ in $x_1,\dots,x_n$. Then  the following statements hold: \\
			1) $S$ is an integral domain.\\
			2) For every $i=1,\dots,n$, there exists a unique endomorphism $\s_i:R\to R$ such that  $\operatorname{lt}(x_ir)=\s_i(r)x_i$ for all $r\in R$.  
		\end{proposition}
		\begin{proof}
			1) This is a simple consequence of  the equality $\operatorname{le}(fg)=\operatorname{le}(f)+\operatorname{le}(g)$, where $f,g\in S\setminus\{0\}$.\\
			2) For every nonzero $r\in R$,   we have $\operatorname{lt}(x_ir)=r'x_i$ for a unique $r'\in R\setminus\{0\}$ because $\operatorname{le}(x_ir)=\operatorname{le}(x_i)+\operatorname{le}(r)$ and $\operatorname{le}(r)=(0,\dots,0)$. This gives a map $\s_i:R\to R$ with $\s_i(0)=0$ such that $\operatorname{lt}(x_ir)=\s_i(r)x_i$ for all $r\in R$. It is clear that $\s_i$ is additive. Comparing the leading terms of $x_i (r_1r_2)$ and $(x_ir_1)r_2$ shows that $\s_i$ is an endomorphism.  
		\end{proof}
		The endomorphisms $\s_i$'s whose existence is established in the proposition will be called the \textit{structural endomorphisms} of $S$. 	

	\section{Amitsur--Small Extensions}\label{sec: amitsur-small}
	This section introduces the concept of 	Amitsur--Small extensions and presents a class of such extensions. 
	\subsection{Amitsur--Small Extensions}
	In this part, we introduce the class of Amitsur--Small extensions, which is a subclass of skew polynomial rings in several variables. Let us fix some notation. Let $R$ be a ring and  $S=R[x_1,\dots,x_n;<]$ be a skew polynomial ring over $R$ in $x_1,\dots,x_n$. Let $\s_i:R\to R$ be the endomorphism corresponding to $x_i$ (see Proposition \ref{prop: simple-proerties-skew}). For  $\mathbf{i}=(i_1,\dots,i_n)\in\mathbb{N}^n$, we denote the map $\s_1^{i_1}\cdots \s_n^{i_n}:R\to R$ by $\s^{\mathbf{i}}$. For a subset \(A\) of \(R\) and \(r\in R\), we set 
	\[
		(A:r)\colonequals\{s\in R\,|\,sr\in A \}.
	\]
		\begin{definition}\label{def: amitsur-small-extensions}
			The skew polynomial ring $S=R[x_1,\dots,x_n;<]$ is said to be an \emph{Amitsur--Small extension} of $R$ if it satisfies the following property:
			For any collection $\{I_{\mathbf{i}}\}_{\mathbf{i}\in \mathbb{N}^n}$ of left ideals of $R$ satisfying  $\s^{\mathbf{j}}({I}_{\mathbf{i}})\subseteq $  $I_{\mathbf{i+j}}$ for all $\mathbf{i},\mathbf{j}\in \mathbb{N}^n$, there exists a nonzero element $r\in R\setminus\{0\}$ that  is not left invertible, 
			$rS\subseteq Sr$, and 
			\((I_{\mathbf{i}}:\s^{\mathbf{i}}(r))\subseteq (rI_{\mathbf{i}}:\s^{\mathbf{i}}(r))\) for all $\mathbf{i}\in \mathbb{N}^n$. 
		\end{definition}
		As an example of an Amitsur--Small extension, let $D$ be a division ring and $x$ be a central indeterminate. Then  the polynomial ring $D[x][t_1,\dots,t_n]$ over $D[x]$  in $n$  central indeterminates is an Amitsur--Small extension of $D[x]$. This result is a consequence of Lemma A and Lemma B  in Amitsur and Small's paper \cite{AmitsurSmall}. Our main objective is to give a generalization of this fact (see Theorem \ref{thm: A-Inv}). We note that an arbitrary Ore extension of $D[x]$ may not be  an Amitsur--Small extension of $D[x]$. As an example, 	the Weyl algebra $k[x][t;id,\frac{d}{d x}]$ over a commutative field $k$  of characteristic zero  is not an Amitsur--Small extension of	$k[x]$ (see Example \ref{exam: notAS}).
		
		The concept of Amitsur--Small extensions is justified by the following result, whose proof is  similar to that of Lemma C in  \cite{AmitsurSmall}.
		\begin{proposition}\label{prop: ASextensions}
			Let $S=R[x_1,\dots,x_n;<]$ be an Amitsur--Small extension of $R$. Then for every maximal left ideal $I$ of $S$, we have $I\cap R\neq 0$.
		\end{proposition}
		\begin{proof}
			For  $\mathbf{i}\in \mathbb{N}^n$, we set 
			 $$I_{\mathbf{i}}=\{0\}\cup \{r\in R \,|\, \exists f\in I \text{ s.t. } \operatorname{lt}(f)=r{\mathbf{x}}^{\mathbf{i}}\}.$$ It is easy to see that each $I_{\mathbf{i}}$ is  a left ideal of $R$, and moreover, 
			$\s^{\mathbf{j}}(I_{\mathbf{i}})\subseteq I_{\mathbf{i}+\mathbf{j}}$ for all $\mathbf{i},\mathbf{j}$. Since $S$ is an Amitsur--Small extension of $R$, there exists $r_0\in R\setminus\{0\}$ such that  \(r_0\) is not left invertible, $r_0S\subseteq Sr_0$, and 
			\[
				(I_{\mathbf{i}}:\s^{\mathbf{i}}(r_0))\subseteq (r_0I_{\mathbf{i}}:\s^{\mathbf{i}}(r_0))
			\] for all $\mathbf{i}$. If $r_0\in I$, we are done. Assume that $r_0\notin I$. Since $I$ is maximal as a left ideal, there exists $$Q=a{\mathbf{x}}^{\mathbf{m}}+\sum_{\mathbf{i}<\mathbf{m}}a_{\mathbf{i}}{\mathbf{x}}^{\mathbf{i}}\in S$$ with $\operatorname{le}(Q)=\mathbf{m}$ as small as possible such that the element
			$$Qr_0-1=a\s^{\mathbf{m}}(r_0){\mathbf{x}}^{\mathbf{m}}+\sum_{\mathbf{i}<\mathbf{m}}a'_{\mathbf{i}}{\mathbf{x}}^{\mathbf{i}}$$ belongs to $I$.   We claim that $\mathbf{m}=(0,\dots,0)$. Assume, on the contrary, that $\mathbf{m}\neq (0,\dots,0)$. Since $a\s^{\mathbf{m}}(r_0)\in I_\mathbf{m}$, that is, \(a\in (I_\mathbf{m}:\s^{\mathbf{m}}(r_0)) = (r_0I_{\mathbf{m}}:\s^{\mathbf{m}}(r_0))\), we see that  $a\s^{\mathbf{m}}(r_0)=r_0b$ for some $b\in I_\mathbf{m}$. There exists an element $P\in I$ such that 
			$$P=b{\mathbf{x}}^{\mathbf{m}}+\sum_{x^\mathbf{i}<x^\mathbf{m}}b_{\mathbf{i}}{\mathbf{x}}^{\mathbf{i}}.$$
			Since $r_0S\subseteq Sr_0$, we can find $P_1\in S$ such that   $r_0P=P_1r_0\in I$. It is easy to check that $\operatorname{lt}(P_1)=a{\mathbf{x}}^{\mathbf{m}}$. 
			It follows that $$(Q-P_1)r_0-1\in I.$$  We have $\operatorname{le}(Q-P_1)<\operatorname{le}(Q)$, contradicting the choice of $Q$. This completes the proof of the claim. Therefore, \(Q =a \in R\). Since $r_0$ is not left invertible in $R$, the element $ar_0-1$ is a nonzero element in $I\cap R$, completing the proof.   
		\end{proof}
	We now give an example of a skew polynomial ring which is not an Amitsur--Small extension. 
		\begin{example}\label{exam: notAS}
		Consider the  Weyl algebra $S = k[x][t;id,\frac{d}{d x}]$ over a commutative field $k$ of zero characteristic as a skew polynomial ring over \(R = k[x]\). The left ideal \(I\) of \(S\) generated by \(t\) is a maximal left ideal. Since \(I\cap R=0\), by  Proposition \ref{prop: ASextensions}, \(S\) is not an Amitsur--Small extension of \(R\). 
	\end{example}

	\subsection{A Class of Amitsur--Small Extensions}
			In this part, we present a family of Amitsur--Small extensions, which requires  a technical lemma. We refer the reader to Cohn's book \cite{Freeidealrings} for the terminology used in the following lemma. 
		\begin{lemma}\label{lem: LemmaA-Inv}
		Let $R$ be a principal (right and left) ideal domain (PID), and let $0\neq r_0\in R$. If 
		$(Rr_0:r)\nsubseteq (rRr_0:r)$
		for some Inv-atom $r\in R$, then $r$ is the right bound of some irreducible left factor of $r_0$. In particular, there exist at most finitely many such elements $r$ up to right associates for any $0\neq r_0\in R$. 
	\end{lemma}
	\begin{proof}
		Let an Inv-atom $r\in R$ satisfy 
		$(Rr_0:r)\nsubseteq (rRr_0:r).$
		Then  we have
		$r_1r=r_2r_0$ for some \(r_1\in R\) and $r_2\in R\setminus rR$. Since  the set 
		$\{s\in R\, |\, r_2s\in rR\}$
		is a right ideal of $R$, we have
		$$\{s\in R\, |\, r_2s\in rR\}=r_3R$$ for some $r_3\in R$. Note that $r_3$ cannot be a unit since otherwise we would have $r_2\in rR$.
		Both $r$ and $r_0$ belong to $r_3R$ because \(Rr=rR\) and \(r_2r_0 = r_1r\in rR\). Therefore, there exists an irreducible element $r'$ that is a left factor of both
		$r$ and $r_0$. Since $r$ is an Inv-atom, we see that $r$ is the right bound of the left irreducible factor $r'$ of $r_0$, i.e.,  $rR=ann(R/r'R)$, completing the proof.    
	\end{proof}
		\begin{remark}
			Lemma \ref{lem: LemmaA-Inv} can be regarded as a generalization of Lemma A in  Amitsur and Small's paper \cite{AmitsurSmall}. 
		\end{remark}
	We are now ready to prove the main result of this section. 
		\begin{theorem}\label{thm: A-Inv}
			Let $R$ be a PID, and let 
			$S=R[x_1,\dots,x_n;<]$ be a skew polynomial ring over $R$ in $x_1,\dots,x_n$  such that all structural endomorphisms of \(S\)  are automorphisms. 
			 Assume that  $R$ has  infinitely many distinct maximal ideals $Rr=rR$ such that $rS\subseteq Sr$. 
			Then  $S$ is an Amitsur--Small extension of $R$, and consequently, for every maximal left ideal \(M\) of \(S\), \(M\cap R\) is a maximal left ideal of \(R\).     
		\end{theorem}
		\begin{proof}
			Let us first note that  the assumption $rS\subseteq Sr$, where \(Rr\) is a maximal ideal, implies that \(R\s_i(r) = Rr\) for all \(i\). Let  $I_{\textbf{i}}=Ra_{\textbf{i}}$, where $\textbf{i}\in \mathbb{N}^n$, be a family of left ideals of $R$ satisfying  $\s^{\textbf{j}}(I_{\textbf{i}})\subseteq I_{\textbf{i}+\textbf{j}}$ for all $\textbf{i},\textbf{j}\in \mathbb{N}^n$. We need to show that  there exists $r\in R\setminus\{0\}$ such that  \(r\) is not left invertible, 
			$rS\subseteq Sr$, and 
			\((I_{\mathbf{i}}:\s^{\mathbf{i}}(r))\subseteq (rI_{\mathbf{i}}:\s^{\mathbf{i}}(r))\) for all $\mathbf{i}\in \mathbb{N}^n$. Since $R$ is a PID,  we can find finitely many $\textbf{i}_1,\dots,\textbf{i}_m$ with the property that 
			for every $\textbf{i}$, there exists some $1\leq l\leq m$ such that
			 $\textbf{i}- \textbf{i}_l\in\mathbb{N}^n$ and $I_{\textbf{i}}=$ $\s^{\textbf{i}-\textbf{i}_l}$ $(I_{\textbf{i}_l})$.
			It follows from the assumption and Lemma \ref{lem: LemmaA-Inv} that  there exists a  maximal ideal  $Rr_0=r_0R$ in \(R\) such that  
			$r_0S\subseteq Sr_0$ and \((I_{\textbf{i}}:r_0)\subseteq (r_0I_{\textbf{i}}:r_0)\) 	for all $\textbf{i}=\textbf{i}_1,\dots,\textbf{i}_m$. Moreover, by Lemma \ref{lem: LemmaA-Inv}, $r_0$ can be chosen such that $r_0$ is not the right bound of any irreducible left factor of $a_{\textbf{i}}$ for all $\textbf{i}=\textbf{i}_1,\dots,\textbf{i}_m$.  We claim that $(I_{\textbf{i}}:r_0) \subseteq (r_0I_{\textbf{i}}:r_0)$ for all $\textbf{i}$. Assume, on the contrary, that there exists some $I_{\textbf{i}}$ which does not satisfy the condition. Then  there exists some $1\leq l\leq m$ such that
		$\textbf{i}- \textbf{i}_l\in\mathbb{N}^n$ and $I_{\textbf{i}}=$ $\s^{\textbf{i}-\textbf{i}_l}$ $(I_{\textbf{i}_l})$.
		 By Lemma \ref{lem: LemmaA-Inv}, we have 
			$r_0R=ann(R/r'R)$ for a left irreducible factor $r'$ of $\s^{\textbf{i}-\textbf{i}_l}$ $(a_{\textbf{i}_l})$. We have
			\[
				r_0R=\s^{-\textbf{i}+\textbf{i}_l}(r_0)R=ann(R/ \s^{-\textbf{i}+\textbf{i}_l}(r')R),
			\]  
			that is, $r_0$ is the right bound of the irreducible left factor $\s^{-\textbf{i}+\textbf{i}_l}$ $(r')$ of $a_{\textbf{i}_l}$. This contradicts the choice of $r_0$. Therefore,  $(I_{\textbf{i}}:r_0) \subseteq (r_0I_{\textbf{i}}:r_0)$ for all $\textbf{i}$. 	Since \(R\s^{\textbf{i}}(r_0) = Rr_0\), we see that 
			$$(I_{\textbf{i}}:\s^{\textbf{i}}(r_0)) \subseteq (r_0I_{\textbf{i}}:\s^{\textbf{i}}(r_0))$$ 
			for all $\textbf{i}$. This completes the proof. The second statement follows from Proposition \ref{prop: ASextensions}. 
		\end{proof}

	\section{Examples of Amitsur--Small Extensions}\label{sec: examples}
		This section presents some examples of Amitsur--Small extensions and establishes a skew Amitsur--Small Theorem. 
		We begin with the one-variable case.  
	\begin{proposition}\label{prop: AS_R[x,s,d]}
		Let $R$ be a PID, $\s:R\to R$ be an automorphism, and $\delta:R\to R$ be a $\s$-derivation. If $R$ has infinitely many distinct maximal ideals invariant under $\s$ and $\delta$, then the Ore extension $R[x;\s,\delta]$ is an Amitsur--Small extension of \(R\), and consequently, for every maximal left ideal \(M\) of $R[x;\s,\delta]$, we have \(M\cap R\neq 0\). 
	\end{proposition}
	\begin{proof}
		By Theorem \ref{thm: A-Inv}, we only need to show that \(R\) has infinitely many maximal ideals $Rr=rR$ such that $rR[x;\s,\delta]\subseteq R[x;\s,\delta]r$. This follows immediately from the assumption that \(\s(Rr)\subseteq Rr\) and \(\delta(Rr)\subseteq Rr\) for  infinitely many distinct maximal ideals \(Rr\) in \(R\).  
	\end{proof}
	
	\begin{remark}
		In the case where $R$ is commutative, the second conclusion of the proposition holds under the weaker assumption that $R$ is a Dedekind domain, as proved by Bavula \cite[Theorem 1.2]{Bavula}.
	\end{remark}
	
	Turning to the several-variable case, we begin with the following result. 
	\begin{proposition}\label{prop: lemmaC}
		Let $R$ be a PID with infinitely many distinct maximal ideals. Then the polynomial ring $R[x_1,\dots,x_n]$ in central indeterminates is an Amitsur--Small extension of \(R\), and consequently, for every maximal left ideal $M$ of $R[x_1,\dots,x_n]$, we have $M\cap R\neq 0$. 
	\end{proposition}
	\begin{proof}
		For any term order \(<\), $R[x_1,\dots,x_n] = R[x_1,\dots,x_n;<]$ is a skew polynomial ring in \(x_1,\dots,x_n\) with respect to \(<\). The result now follows from Theorem \ref{thm: A-Inv}. 
	\end{proof}
	\begin{remark}
		The converse of the proposition holds true for commutative rings. More precisely, if $M\cap R\neq 0$ for every maximal ideal $M$ of the polynomial ring $R[x_1,\dots,x_n]$ over a commutative PID $R$ in central indeterminates,  then $R$ has infinitely many maximal ideals. One can prove this result using the language of G-rings. For more details, see \cite[Theorem 27 and Theorem 147]{kaplansky2006commutative}.
	\end{remark}

	To present a skew version of the Amitsur--Small Theorem, we need the following result. 
	\begin{proposition}\label{prop: center-maximal-ideals}
		Let $D$ be a division ring, $\s:D\to D$ an automorphism and $\delta:D\to D$  a $\s$-derivation. Then  the following statements are equivalent:\\
		(1) The center of the Ore extension $D[x;\s,\delta]$ contains a nonconstant polynomial.\\
		(2) $D[x;\s,\delta]$ has infinitely many distinct maximal ideals. 
	\end{proposition}
	\begin{proof}
		(1)\(\implies\)(2): Let \(h_0(x)\) be a nonconstant polynomial of the least degree in the center of \(D[x;\s,\delta]\) . Let \(q(x)\) be a polynomial in $D[x;\s,\delta]$ of the least degree such that \(D[x;\s,\delta]q(x) = q(x) D[x;\s,\delta]\). By Cauchon's description of the ideal structure of \(D[x;\sigma,\delta]\) (see \cite{Cauchon}), every ideal of $D[x;\s,\delta]$ is generated by a nonconstant polynomial \(h(x)q(x)^n\), where  \(h(x)\in C(D)_{\sigma,\delta}[h_0(x)]\) and \(n\geq 0\). Here,   \(C(D)_{\sigma,\delta}\) is the field
		\[
			C(D)_{\sigma,\delta} = \{a\in D: \sigma(a) = a, \delta(a) =0, \forall b\in D\quad ab = ba\}.
		\]
		It follows that for every irreducible polynomial  \(g(t)\in C(D)_{\sigma,\delta}[t]\), the ideal of \(D[x;\sigma,\delta]\) generated by \(g(h_0(x))\) is maximal. It is well known that the polynomial ring \(F[t]\) over a commutative field \(F\) has infinitely many monic irreducible polynomials, from which (2) follows. \\
		(2)\(\implies\)(1): If the center of \(D[x;\sigma,\delta]\) does not contain a nonconstant polynomial, then by Cauchon's result, every ideal of $D[x;\s,\delta]$ is generated by a polynomial \(q(x)^n\) for some \(n\geq 1\), in which case	\(D[x;\sigma,\delta]\) would have a single maximal ideal. This proves the implication (2)\(\implies\)(1).  
	\end{proof}
	\begin{remark}
		It is easy to prove that if the center of \(D[x;\sigma,\delta]\) contains a nonconstant polynomial, then \(\sigma\) has a finite inner order and \(\delta\) is quasi-algebraic. For more details, see \cite{lam1invariant} and the references therein.   
	\end{remark}
	Let $D$ be a division ring. Let $\s:D\to D$ be an automorphism  and $\delta:D\to D$ be   a $\s$-derivation such that \(\sigma\delta = \delta\sigma\). Consider the ring \(D[x_1,\dots,x_n;\sigma,\delta]\), generated over \(D\) by \(x_1,\dots,x_n\), subject to the relations \(x_ix_j=x_jx_i\), \(x_ia=\s(a)x_i+\delta(a)\) for all \(i,j\) and \(a\in D\). It can be shown that  \(D[x_1,\dots,x_n;\sigma,\delta]\) is an iterated skew polynomial ring. More precisely, we have 
	\[
		D[x_1,\dots,x_n;\sigma,\delta] = D[x_1;\s_1;\delta_1][x_2;\s_2;\delta_2]\dots [x_n;\s_n;\delta_n],
	\] 
	where \(\s_i|_R =\s, \delta_i|_R = \delta,  \s_j(x_i) = x_i,  \delta_j(x_i) = 0\) for all \(1\leq i < j \leq n\). See also \cite[Theorem 4.2]{voskoglou1986extending}. 
	\begin{theorem}\label{thm: skew-AS}
		Let $D$ be a division ring. Let $\s:D\to D$ be an automorphism  and $\delta:D\to D$ be   a $\s$-derivation such that \(\sigma\delta = \delta\sigma\).   If the center of $D[x;\s,\delta]$ contains a nonconstant polynomial, then any simple  \(D[x_1,\dots,x_n;\sigma,\delta]\)-module is finite-dimensional as a vector space over $D$.  
	\end{theorem}

	\begin{proof}
		The discussion preceding the theorem shows that  \[S = D[x_1,\dots,x_n;\sigma,\delta] = D[x_1;\sigma,\delta][x_2,\dots,x_n; <_l] \]	
		is a skew polynomial ring over \(D[x_1,\sigma, \delta]\) in variables \(x_2,\dots,x_n\), where \(<_l\) denotes  the lexicographical order  with $x_2<...<x_n$ (see Example \ref{exam: iterated-skew}). By Proposition \ref{prop: center-maximal-ideals}, \(D[x_1;\sigma,\delta]\) has infinitely many maximal ideals. As shown in the proof fo the proposition, every maximal ideal of \(D[x_1;\sigma,\delta]\) is generated by a central polynomial. Therefore, there are infinitely many maximal ideals \(D[x_1;\sigma,\delta]f\) in \(D[x_1,\sigma, \delta]\) such that \(fS\subset Sf\). It follows from Theorem \ref{thm: A-Inv} that \(S\) is an Amitsur--Small extension of \(D[x_1;\sigma,\delta]\). In particular, for every maximal left ideal \(M\) in \(S\), we have \(M\cap D[x_1;\sigma,\delta]\neq 0\). By symmetry, for every maximal left ideal \(M\) in \(S\), we have \(M\cap D[x_i;\sigma,\delta]\neq 0\) for all \(i\). Interpreting this result in the context of \(D[x_1,\dots,x_n;\sigma,\delta]\)-modules yields the desired conclusion. 
	\end{proof}

	As a special case, we record the following result.
	\begin{corollary}
		Let $F$ be a finite field and $\s:F\to F$ be an automorphism. Then  every simple \(F[x_1,\dots,x_n;\sigma]\)-module is finite-dimensional as a vector space over \(F\).  
	\end{corollary}
	\begin{proof}
		It is clear that \(\s\) has a finite order, say \(m\). Then \(x^m\) belongs to the center of \(F[x;\sigma]\). By Theorem \ref{thm: skew-AS}, every simple \(F[x_1,\dots,x_n;\sigma]\)-module is finite-dimensional as a vector space over \(F\).
	\end{proof}


\section*{Acknowledgments}
I would like to thank A.~Leroy~for proposing a sketch of the proof of Proposition \ref{prop: center-maximal-ideals}.

\bibliographystyle{plain}
\bibliography{SNbiblan}

 \end{document}